\documentclass[12pt,oneside,reqno]{amsart}
\usepackage[all]{xy}
\usepackage{amsfonts,amsmath,oldgerm,amssymb,amscd}
\usepackage{enumerate}
\usepackage{comment}  % added by Nikos
\usepackage{soul}     % added by Nikos
\usepackage{xcolor}    % added by Nikos
\usepackage[normalem]{ulem}
\usepackage{longtable}
\definecolor{mycolor}{rgb}{0.100, 0.300, 0.950}
\UseComputerModernTips
\numberwithin{equation}{section}

\usepackage{fancyhdr}

\usepackage{xcolor}

\usepackage{caption} 
\usepackage[breaklinks]{hyperref}

\newtheorem{theorem}{Theorem}
\newtheorem{proposition}[theorem]{Proposition}
\newtheorem{conj}{Conjecture}

\theoremstyle{remark}
\newtheorem{remark}[theorem]{Remark}

\newcommand{\al}{\alpha}
\newcommand{\be}{\beta}

\newcommand{\fn}{\mathfrak{n}}
\newcommand{\fnp}{\mathfrak{n}^\prime}

\newcommand{\Z}{\mbox{$\mathbb Z$}}
\newcommand{\Q}{\mbox{$\mathbb Q$}}
\newcommand{\Fl}{\mathbb{F}_\ell}

\newcommand{\Da}{\Delta_{a}}
\newcommand{\Db}{\Delta_b}

\newcommand{\Gal}{\operatorname{Gal}}
\newcommand{\brhoE}{\bar{\rho}_{E,n}}
\newcommand{\brhoEi}{\bar{\rho}_{E_{i,k}}}
\newcommand{\brhoFi}{\bar{\rho}_{F_{i,k}}}
\newcommand{\brhof}{\bar{\rho}_{f,\fn}}
\newcommand{\brhog}{\bar{\rho}_{g,\fnp}}

\newcommand\raisepunct[1]{\,\mathpunct{\raisebox{0.5ex}{#1}}} % added by Nikos

\DeclareMathOperator{\Frob}{Frob}
\DeclareMathOperator{\GL}{GL}
\DeclareMathOperator{\Tr}{Tr}

\begin{document}
\title{Perfect powers in sum of three fifth powers} %\\ \today}
\author[Das]{Pranabesh Das}

\address{Pure Mathematics, University of Waterloo, 200 University Avenue West, 
Waterloo, Ontario, Canada
N2L 3G1 }
\email{pranabesh.math@gmail.com}

\author[Dey]{Pallab Kanti Dey}

\address{Department of Mathematics, Ramakrishna Mission Vivekananda Educational and Research Institute, Belur Math, Howrah - 711201, West Bengal, India.}
\email{pallabkantidey@gmail.com}

\author[Koutsianas]{Angelos Koutsianas}

\address{Department of Mathematics, University of British Columbia\\
1984 Mathematics Road, Vancouver, BC V6T 1Z2, Canada.} 
\email{akoutsianas@math.ubc.ca}

\author[Tzanakis]{Nikos Tzanakis}
\address{Department of Mathematics \& Applied Mathematics,
University of Crete, Voutes Campus, 70013 Heraklion, Greece}
\email{tzanakis@math.uoc.gr}

\maketitle
\pagenumbering{arabic}

\begin{abstract}
In this paper we determine the perfect powers that are sums of three fifth powers in an arithmetic progression. More precisely, we completely solve the Diophantine equation 
$$
(x-d)^5 + x^5 + (x + d)^5 = z^n,~n\geq 2,
$$
where $d,x,z \in \mathbb{Z}$ and $d = 2^a5^b$ with $a,b\geq 0$.  
\end{abstract}

\section{Introduction}

In 1956, Sch\"{a}ffer \cite{Schaffer56} considered the equation 
\begin{equation}\label{eq0a}
1^k + 2^k +\cdots + x^k =y^n.
\end{equation}
He proved that if $k\geq 1$ and $n\geq 2$ are fixed, then \eqref{eq0a} has only finitely many solutions except for the cases $(k,n) \in \{(1,2), (3,2), (3,4), (5,2)\}$. In the same paper Sch\"{a}ffer stated the following conjecture on the integral solution of \eqref{eq0a}.
\begin{conj}{\bf [Sch\"{a}ffer, \cite{Schaffer56}]}\label{Main_Conj} \\
Let $k\geq 1,n\geq 2$ be integers and $(k,n) \notin \{(1,2), (3,2), (3,4), (5,2)\}$. The equation
\begin{equation*}
1^k + 2^k +\cdots + x^k =y^n,
\end{equation*}
has only one non-trivial solution, namely $(k,n,x,y) = (2,2,24,70)$.
\end{conj}

The equation \eqref{eq0a} and its generalizations have a long and rich history. Bennett, Gy\H{o}ry, Pint\'{e}r \cite{BennettGyoryPinter04} proved the Conjecture \ref{Main_Conj} for arbitrary $n$ and $k\leq 11$. Following and extending the approach of \cite{BennettGyoryPinter04}, Pint\'{e}r \cite{ArgaezPatel19} proved Conjecture \ref{Main_Conj} for odd values of $k$ with $1 \leq k\leq 170$ and even values of $n$.

As a natural generalization of equation \eqref{eq0a}, Zhang and Bai \cite{ZhangBai13} considered the equation
\begin{equation}\label{eq0c}
(x+1)^{k}+(x+2)^{k}+\cdots+(x+r)^{k}=y^{n}.
\end{equation}
The equation \eqref{eq0c} is comparatively more difficult than the equation \eqref{eq0a}. There has been some progress on equation \eqref{eq0c} for particular values of $k$, $n$ and $r$. Below we list down all results to the best of our knowledge obtained for equation \eqref{eq0c} .

We first list down the results obtained for power sums of a range of values of $r$. Stroeker \cite{Stroeker95} completely solved the equation \eqref{eq0c} for $k=3$, $n=2$ and $2\leq r\leq 50$. Recently, Bennett, Patel and Siksek \cite{BennettPatelSiksek17} extended the result of Stroeker for $n \geq 3$. Zhang and Bai \cite{ZhangBai13} solved the equation \eqref{eq0c} for $k=2$ and $r=x$. Bartoli and Soydan \cite{Soydan17,BartoliSoydan20} extended the result of Zhang and Bai \cite{ZhangBai13} for $k \geq 2$ and $r = lx$ with $l \geq 2$.

The equation \eqref{eq0c} has also been studied for fixed $r$. Cassels \cite{Cassels85} solved the equation \eqref{eq0c} completely for $n = 2$, $r = 3$ and $k = 3$. Zhang \cite{Zhang14} subsequently considered the equation \eqref{eq0c} for $r = 3$ and he solved it completely for $k \in \{2,3,4 \}$. Recently, Bennet, Patel and Siksek \cite{BennettPatelSiksek16} extended Zhang's result, by completely solving  equation \eqref{eq0c} for $r = 3$ in the cases $k = 5$ and $k = 6$. Several authors have also studied equations \eqref{eq0a}, \eqref{eq0c} and its variants using a variety of classical and modern techniques (see e.g. \cite{BennettGyoryPinter04, BerczesPinkSavasSoydan18, Brinza84, DasDeyMajiRout20, DasDeyRout20, GyoryTijdemanVoorhoeve80, Hajdu15, JacobsonPinterWalsh03}).

As a natural generalization of all the above results many mathematicians have recently studied power sums in arithmetic progression. They have considered the equation
\begin{equation}\label{eq0d}
(x+d)^k + (x+2d)^k + \cdots + (x+rd)^k = y^n, \ x,y,d\in \mathbb{Z},\ r,k \in \mathbb{Z}_{\geq0},\ n\ \text{prime}.
\end{equation}

In this paper we are particularly interested in the case $r =3$, in particular for the equation
\begin{equation}\label{eq0e}
(x-d)^k + x^k + (x+d)^k = y^n, \quad n\geq 2.
\end{equation}

We mention some related results for $k\leq 4$. Koutsianas \cite{Koutsianas19} studied the equation \eqref{eq0e} for $k = 2$, where $d$ is of the form $p^b$ with $p$ a suitable prime. Koutsianas and Patel \cite{KoutsianasPatel18} completely solved the equation \eqref{eq0e} for $k = 2$ and for all values of $1 \leq d \leq 5000$, using the characterization of primitive divisors in Lehmer sequences by Bilu-Hanrot-Voutier \cite{BiluHanrotVoutier01}. For $1 \leq d \leq 10^6$ and $k=3$,  Arg\'{a}ez-Garc\'{i}a and Patel \cite{ArgaezPatel19, ArgaezPatel20} and Arg\'{a}ez-Garc\'{i}a \cite{Argaez19} studied the equation \eqref{eq0d} for $r=3$, $7$ and $r=5$, respectively. For $k=4$ the equation \eqref{eq0e} was solved by Zhang \cite{Zhang17} for some particular choices of $d$ and Langen \cite{Langen19} under the assumption $\gcd(x,d)=1$.  For further reference we include all known results on equation \eqref{eq0d} in Table \ref{tab0}.
\begin{center}
\begin{table}[h]
 \begin{tabular}{ || p{3cm} | c | c | c | p{6cm} ||}
    \hline \hline
    $d$ & $r$ & $k$ & $n$ & References \\ \hline \hline
    $1$ & $x$ & $2$ & $\geq 2$ & Zhang and Bai \cite{ZhangBai13}  \\ \hline
    $1$ & $3$ & $3$ & $2$ & Cassels \cite{Cassels85}\\ \hline
    $1$ & $3$ & $\{2,3,4 \}$ & $\geq 2$ & Zhang \cite{Zhang14}\\ \hline
    $1$ & $3$ & $\{5,6 \}$ & $\geq 2$ & Bennett, Patel and Siksek \cite{BennettPatelSiksek16}\\ \hline
    $1$ & $\{1,\cdots,50\}$ & $3$ & $\geq 2$ & Bennett, Patel and Siksek \cite{BennettPatelSiksek17}\\ \hline
    $1$ & $\{2,\cdots,10\}$ & $2$ & $\geq 2$ & Patel \cite{Patel18}\\ \hline
    a suitable set of prime powers & $3$ & $2$ & $\geq 7$ & Koutsianas \cite{Koutsianas19}\\ \hline
    composed of a suitable set of prime powers& $3$ & $4$ & $\geq 11$ & Zhang \cite{Zhang17}\\ \hline
    $\gcd(x,d)=1$ & $3$ & $4$ & $\geq 2$ & Langen \cite{Langen19}\\ \hline 
    $\{1,\cdots, 10^4\}$ & $\{2,\cdots,10\}$ & $2$ & $\geq 2$ & Kundu and Patel \cite{KunduPatel19}\\ \hline
    $\{1,\cdots, 10^6\}$ & $3,7$ & $3$ & $\geq 5$ & Arg\'{a}ez-Garci\'{a} and Patel \cite{ArgaezPatel19, ArgaezPatel20}\\ \hline
    $\{1,\cdots, 10^6\}$ & $5$ & $3$ & $\geq 5$ & Arg\'{a}ez-Garci\'{a} \cite{Argaez19}\\ \hline
    $\{1,\cdots, 5000\}$ & $3$ & $2$ & $\geq 2$ & Koutsianas and Patel \cite{KoutsianasPatel18}\\ \hline\hline
   \end{tabular}
  \caption{Notable results on solutions of special cases of \eqref{eq0d}.} 
\end{table}\label{tab0}
\end{center}

The general equation \eqref{eq0e} for $k\geq 5$ is a difficult problem. In this paper, we study the Diophantine equation
\begin{equation}\label{eq:main_k_5}
(x - d)^5 + x^5 + (x + d)^5 = z^n, \quad n\geq 2,~xz\neq 0.
\end{equation}

Recently, Bennett and Koutsianas \cite{BennettKoutsianas20} solved equation \eqref{eq:main_k_5} with the natural assumption $\gcd(x,d)=1$. This assumption enables them to factorize the left-hand side of \eqref{eq:main_k_5} and reduce the problem to the resolution of Fermat type equations of signature $(n,n,2)$ with coefficients independent of $d$. In the general case, the coefficients of the Fermat type equations have prime factors that divide $10\cdot\gcd(x,d)$. Therefore, without any restrictions to $\gcd(x,d)$ we are not able to solve \eqref{eq:main_k_5} with the current techniques. The existence of infinite family of solutions for small exponents $n$, for instance the solutions $(x,d,y,n) = (ra^4, sa^4, a^3, 7),$  $(ra^6, sa^6, a, 31)$ and $(ra^8, sa^8, a, 41)$ where $r,s\in\Z^*$ and\footnote{For example, the quadruples $(x,d,y,n)=(33^4, 33^4, 33^3, 7)$, $(2\cdot{276}^6, {276}^6, 276, 31)$ and $(243^8,2\cdot 243^8,243,41)$ are the solutions to the equation \eqref{eq0e} for the three smallest positive values of $a$.} $a=(r-s)^5+r^5+(r+s)^5$, show that the complete resolution of \eqref{eq:main_k_5} is a hard and challenging problem.

From the above we understand that if we want to study the resolution of \eqref{eq:main_k_5} without the assumption $\gcd(x,d)=1$ and having a freedom of the choice of $d$ we have to fix the prime factors of $d$. Let assume that $d$ is divisible by primes that lie in a fixed finite set $S$. As we mentioned in the previous paragraph the coefficients of the Fermat type equations will depend on $S\cup\{2,5\}$. Because the primes $2$ and $5$ always show up for any choice of $S$ it is very natural to study the resolution of \eqref{eq:main_k_5} for $S=\{2,5\}$, i.e. when $d=2^a5^b$ with $a,b\geq 0$. We prove the following which is a generalization of \cite[Theorem 1]{BennettPatelSiksek16} for the case $\gcd(x,d)\geq 1$.

\begin{theorem} \label{Thm1}
Let $n\geq 2$ be an integer and $d = 2^a 5^b$ with integers $a,b \geq 0$. 
Then, the equation
\[
(x - d)^5 + x^5 + (x + d)^5 = z^n,
\]
is solvable in integers $x,z$ with $xz\neq 0$ only if $n=5$ and $a\geq 1$,
in which case all integer solutions are given by  $(x,z) = \pm(d/2,3d/2)$.
\end{theorem}

The paper is organized as follows. In Section \ref{sec:primitive_solutions} we associate to a solution $(x,z)$ of \eqref{eq:main_k_5} and $d=2^a5^b$ two Fermat type equations of signature $(n,n,2)$ with pairwise coprime terms. In Section \ref{sec:small_n} we resolve \eqref{eq:main_k_5} for $n=2$, $3$ and $5$ using a variety of elementary and advanced techniques. In Section \ref{sec:Frey_curves} we explain how we can apply the modular method and the recipes in \cite{BennettSkinner04} to resolve \eqref{eq:main_k_5} when $n\geq 7$ is a prime. Finally, in Section \ref{sec:Proof large n} we complete the proof of Theorem \ref{Thm1} for $n\geq 7$ and we give the necessary details of the computations.

\begin{remark}
In principle, the method we describe in this paper will work for any choice of the set $S$ when $S$ contains $2$ and $5$. However, the required computations for the spaces of newforms are beyond to the current computer power.
\end{remark}

The computations of this paper have been accomplished in computer software \texttt{Magma} \cite{Magma} and the code for all the computations in this paper can be found at
\begin{center}
\url{https://github.com/akoutsianas/5th_powers}.
\end{center}

%%%%%%%%%%%%%%%%%%%%%%%
\section*{Acknowledgements}

The first author wants to thank Professor Cameron Stewart and Professor Yu-Ru Liu  for providing funding to third authors' visit to Waterloo and the Department of Mathematics at the University of Waterloo for the very nice working environment where part of this project was accomplished. The first author also wants to thank Professor Nikos Tzanakis for visiting him to University of Crete, Heraklion and for many fruitful discussions. The second author's research is supported by National Board for Higher Mathematics, India.  The third author is grateful to Professor John Cremona for providing access to the servers of the Number Theory group of Warwick Mathematics Institute where all the computations took place.

%%%%%%%%%%%%%%%%%%%%%%
\section{Preliminaries}\label{sec:primitive_solutions}
We have the equation

\begin{equation}\label{eq general}
(x - d)^5 + x^5 + (x + d)^5 = z^n, \quad d = 2^a 5^b,\; a,b\geq 0,
\quad xz\neq 0.
\end{equation}

Clearly, in order to prove Theorem \ref{Thm1} we may assume that $n$ is a prime number. Because \eqref{eq general} can be rewritten as
\begin{equation}\label{eq:reduce_initial}
x(3x^4 + 20d^2x^2 +10d^4) = z^n,
\end{equation}
it suffices to consider only the case in which both $x$ and $z$ are positive.

Let $\nu_p(N)$ denotes the $p$-adic valuation of an integer $N$, where $p$ is a prime. We set 
\[
x=2^{\alpha}5^{\beta}x_1,\; \gcd(x_1,10)=1,
                   \quad P=3x^4 + 20d^2x^2 +10d^4,
       \quad z=2^u5^vZ,\; \gcd(Z,10)=1,
\]
hence
\begin{equation}   \label{eq xP=nth pow}
2^{\alpha}5^{\beta}x_1P=2^{nu}5^{nv}Z^n.
\end{equation}

Since $\gcd(x,P)=\gcd(x,10d^4)$, it follows that 
\[
\gcd(x,P)=2^{\min\{4a+1,\alpha\}}\cdot 5^{\min\{4b+1,\beta\}}
\quad\text{and}\quad \gcd(x_1,P)=1.
\]
Let us put $P=2^{\nu_2(P)}5^{\nu_5(P)}P_1$. Clearly, $\gcd(x_1,P_1)=1$ and $\gcd(P_1,10)=1$. 
Thus, $P_1=2^{-\nu_2(P)}5^{-\nu_5(P)}P$. Using these in \eqref{eq xP=nth pow} we obtain
$2^{\al+\nu_2(P)}5^{\be+\nu_5(P)}x_1P_1=2^{nu}5^{nv}Z^n$, where $\gcd(x_1P_1Z,10)=1$. 
It follows that 
\[
    \al+\nu_2(P)=nu,\quad \be+\nu_5(P)=nv,
\]
and $x_1P_1=Z^n$. From $\gcd(x_1,P_1)=1$ it follows that
\begin{equation}  \label{eq x1P1=Z^n}
    x_1=z_1^n,\; P_1=z_2^n,\; Z=z_1z_2,\; \gcd(z_1,z_2)=1,\; \gcd(z_1z_2,10)=1.
\end{equation}

We rewrite equation \eqref{eq:reduce_initial} equivalently in the following two ways:
\begin{align}
10(x^2+d^2)^2-7x^4 & = P,\label{identity I}\\
(3x^2+ 10d^2)^2 - 70 d^4 & = 3P. \label{identity II}
\end{align}
Noting that
\begin{equation}\label{eq:P_with_a_b}
P=3\cdot 2^{4\alpha}5^{4\beta}x_1^4+2^{2a+2\alpha+2}5^{2b+2\beta+1}x_1^2 
                     + 2^{4a+1}5^{4b+1},
\end{equation}
we consider four cases according to the values of $a,\alpha$ and $b,\beta$.

\paragraph{\bf Case I} Suppose $4\alpha<4a+1$ and $4\beta<4b+1$. This is 
equivalent to $a\geq \alpha$ and $b\geq \beta$ and from 
\eqref{eq:P_with_a_b} we get $v_2(P) = 4\alpha$ and $v_5(P) = 4\beta$. 
Because $x = 2^{\alpha}5^{\beta}z_1^n$, $d=2^a5^b$ and 
$P=2^{4\alpha}5^{4\beta}z_2^n$ dividing equations \eqref {identity I} 
and \eqref{identity II} by $2^{4\alpha} 5^{4\beta}$ 
we obtain the following two equations
\begin{align}
& {z_2}^n + 7{z_1}^{4n} 
          = 10 ( {z_1} ^{2n} + 2^{2a-2\alpha}5^{2b-2\beta})^2 
                                          \label{mult 1-1},\\
& 3{z_2}^n + 7\cdot 2^{4(a-\alpha)+1} 5^{4(b-\beta)+1}= 
           (3{z_1}^{2n}+ 2^{2a-2\alpha+1} 5^{2b-2\beta+1})^2.
                                             \label{mult 1-2}
\end{align}

\paragraph{\bf Case II} Suppose $4\alpha<4a+1$ and $4\beta>4b+1$. This is 
equivalent to $a\geq \alpha$ and $\beta\geq b+1$ and from 
\eqref{eq:P_with_a_b} we get that $v_2(P)= 4\alpha$ and $v_5(P)= 4b+1$. 
Because 
$x = 2^{\alpha}5^{\beta}z_1^n$, $d=2^a5^b$ and 
$P=2^{4\alpha}5^{4b + 1}z_2^n$ dividing equations \eqref {identity I} 
and \eqref{identity II} by $2^{4\alpha} 5^{4b + 1}$ 
we obtain the following two equations
\begin{align}
& {z_2}^n + 7\cdot 5^{4(\beta-b)-1}{z_1}^{4n} 
    = 2 ( 5^{2\beta-2b}{z_1} ^{2n} + 2^{2a-2\alpha})^2
                                          \label{mult 2-1},\\
& 3{z_2}^n + 7\cdot 2^{4(a-\alpha)+1} 
        = 5(3\cdot 5^{2\beta-2b-1}{z_1}^{2n}+ 2^{2a-2\alpha+1} )^2.
                              \label{mult 2-2}
\end{align}
\paragraph{\bf Case III} Suppose $4\alpha>4a+1$ and $4\beta<4b+1$. This is 
equivalent to $\alpha\geq a+1$ and $b\geq \beta$ and from 
\eqref{eq:P_with_a_b} we get that $v_2(P)= 4a+1$ and $v_5(P)= 4\beta $. 
Because $x = 2^{\alpha}5^{\beta}z_1^n$, $d=2^a5^b$ and 
$P=2^{4a + 1}5^{4\beta}z_2^n$ dividing equations \eqref {identity I} and 
\eqref{identity II} by $2^{4a + 1}5^{4\beta}$ 
we obtain the following two equations
\begin{align}
& {z_2}^n + 7\cdot 2^{4(\alpha-a)-1}{z_1}^{4n} 
        = 5 ( 2^{2\alpha-2a}{z_1} ^{2n} + 5^{2b-2\beta})^2 
                                       \label{mult 3-1},\\
& 3{z_2}^n + 7\cdot 5^{4(b-\beta)+1}
             = 2(3\cdot 2^{2\alpha-2a-1}{z_1}^{2n}+ 5^{2b-2\beta+1})^2. 
                                          \label{mult 3-2}
\end{align}

\paragraph{\bf Case IV} Suppose $4\alpha>4a+1$ and $4\beta>4b+1$. 
This is equivalent to $\alpha\geq a+1$ and $\beta\geq b+1$ and from
\eqref{eq:P_with_a_b} we get that $v_2(P)= 4a+ 1$ and $v_5(P)= 4b+1$. Because 
$x = 2^{\alpha}5^{\beta}z_1^n$, $d=2^a5^b$ and $P=2^{4a + 1}5^{4b+1}z_2^n$, 
dividing equations \eqref {identity I} and \eqref{identity II} by 
$2^{4a + 1}5^{4b+1}$ we obtain the following two equations
\begin{align}
 & z_2^n + 7\cdot 2^{4(\alpha-a)-1} 5^{4(\beta-b)-1}{z_1}^{4n}   
                      = (2^{2\alpha-2a} 5^{2\beta-2b}z_1^{2n} + 1)^2,                                                    \label{mult 4-1} \\
& 3z_2^n + 7 = 10(3\cdot 2^{2\alpha-2a-1} 
5^{2\beta-2b-1}{z_1}^{2n}+1)^2.\label{mult 4-2}
\end{align}

Equations \eqref{mult 1-1}-\eqref{mult 4-2} all have
the general shape of a ternary generalized Fermat-type 
equation of type $(n,n,2)$, namely,
\begin{equation} \label{eq Ben-Skin}
Aa^n + Bb^n = Cc^2 .     
\end{equation}
where $A,a,B,b,C,c$ are shown in Table \ref{tab:AaBbCc}.
We will need to view our equations as such in Sections 
\ref{sec:Frey_curves} and \ref{sec:Proof large n} where we 
treat the case $n\geq 7$ by applying the recipes  of 
\cite{BennettSkinner04}. For the application of the results
therein we need that $Aa,Bb,Cc$ be pairwise relatively prime.
Since we have already seen that $\gcd(z_1z_2,10)=1$ and
$\gcd(z_1,z_2)=1$, the said requirement amounts in proving 
that $z_1,z_2$ and the $c$'s in Table \ref{tab:AaBbCc}
 are not divisible by $7$. This is proved as follows:
First we observe that, if in any equation 
\eqref{mult 1-1}--\eqref{mult 4-2} $z_1$ is divisible by $7$,
then the same is true also for $z_2$, which contradicts 
$\gcd(z_1,z_2)=1$; thus $7\nmid z_1z_2$. Further, it is 
easily checked that, for any $c$ in Table \ref{tab:AaBbCc},
the divisibility of $c$ by $7$ implies that $-1$ is a quadratic
residue $\bmod\,7$ which is absurd.

\begin{table}[!h]
  \centering
     \begin{tabular}{||c|c|c|c|c|c|c||}
        \hline
 Case-Eq. & $A$ & $a$ & $B$ & $b$ & $C$ & $c$ \\
        \hline\hline
(I)-\eqref{mult 1-1} & $1$ & $z_2$ & $7$ & $z_1^4$ & $10$ 
    & $z_1^{2n}+2^{2(a-\al)}\cdot 5^{2(b-\be)}$ 
    \\ \hline
(I)-\eqref{mult 1-2} & $3$ & $z_2$ 
 & $7\cdot 2^{4(a-\al)+1}\cdot 5^{4(b-\be)+1}$ & $1$ & $1$ 
    & $-3z_1^{2n}-2^{2(a-\al)+1}\cdot 5^{2(b-\be)+1}$ 
    \\ \hline
(II)-\eqref{mult 2-1} & $1$ & $z_2$ & $7\cdot 5^{4(\be-b)-1}$ 
     & $z_1^4$ & $2$ 
    & $5^{2(\be-b)}z_1^{2n}+2^{2(a-\al)}$ 
    \\ \hline
(II)-\eqref{mult 2-2} & $3$ & $z_2$ 
 & $7\cdot 2^{4(a-\al)+1}$ & $1$ & $5$ 
    & $-3\cdot 5^{2(\be-b)-1}z_1^{2n}-2^{2(a-\al)+1}$ 
    \\ \hline     
(III)-\eqref{mult 3-1} & $1$ & $z_2$ & $7\cdot 2^{4(\al-a)-1}$ 
     & $z_1^4$ & $5$ 
    & $2^{2(\al-a)}z_1^{2n}+5^{2(b-\be)}$ 
    \\ \hline
(III)-\eqref{mult 3-2} & $3$ & $z_2$ 
 & $7\cdot 5^{4(b-\be)+1}$ & $1$ & $2$ 
    & $3\cdot 2^{2(\al-a)-1}z_1^{2n}+5^{2(b-\be)+1}$ 
    \\ \hline  
(IV)-\eqref{mult 4-1} & $1$ & $z_2$ 
& $7\cdot 2^{4(\al-a)-1}\cdot 5^{4(\be-b)-1}$ 
     & $z_1^4$ & $1$ 
    & $2^{2(\al-a)}\cdot 5^{2(b-\be)}z_1^{2n}+1$ 
    \\ \hline
(IV)-\eqref{mult 4-2} & $3$ & $z_2$ 
 & $7$ & $1$ & $10$ 
    & $3\cdot 2^{2(\al-a)-1}\cdot 5^{2(b-\be)+1}z_1^{2n}+1$ 
    \\ \hline            
\end{tabular}
\caption{Parameters needed for the application of the recipes
    of \cite{BennettSkinner04}.}
    \label{tab:AaBbCc}
\end{table}

%%%%%%%%%%%%%%%%%%%%
\section{Solving equation \eqref{eq general} when $n = 2,3,5$}\label{sec:small_n}

In this section we prove that equation \eqref{eq general} is impossible when 
$n=2,3$ and solve the equation when $n=5$.

\subsection{The case $n=2$}\label{subsec n=2}

In this case equation \eqref{eq general} becomes
\begin{equation}\label{eq initial n=2}
(x - d)^5 + x^5 + (x + d)^5 = z^2, \quad d = 2^a 5^b,\; a,b\geq 0,
\quad xz\neq 0.
\end{equation}
We have to consider each case (I)-(IV) 
(as defined in Section \ref{sec:primitive_solutions}) separately.

Case (I): By \eqref{mult 1-1} we have $z_2^2 + 7z_1^8\equiv 0\pmod{5}$. 
However, from $5\nmid z_1z_2$ it follows that 
$z_2^2 + 7z_1^8\equiv 1,3\pmod{5}$ and we get a contradiction.

Case (II): By \eqref{mult 2-2} we have 
$3z_2^2 + 7\cdot 2^{4(a-\alpha) + 1}\equiv 0\pmod{5}$. 
However, 
$3z_2^2 + 7\cdot 2^{4(a-\alpha) + 1}\equiv 3z_2^2+14\not\equiv 0\pmod{5}$ 
(actually, $3z_2^2+14\equiv \mbox{$1$ or $2\pmod{5}$}$ because $5\nmid z_2$)
and we get a contradiction.

Case (III): By \eqref{mult 3-1} and $2\nmid z_1$ we get 
$z_2^2 + 7\cdot 2^{4(\alpha-a) - 1}\equiv 5\pmod{8}$. On the other hand, from 
$2\nmid z_2$ and $4(\alpha-a) - 1\geq 3$ we have 
$z_2^2 + 7\cdot 2^{4(\alpha-a) - 1}\equiv 1\pmod{8}$ which is a contradiction.

Case (IV): Putting $z_1^2=x_1$ in equation \eqref{mult 4-1} leads us to the 
equation
\[
3\cdot 2^{4(\al-a)-1}5^{4(\be-b)-1}x_1^4+2^{2(\al-a)+1}5^{2(\be-b)}x_1^2+1
={z_2}^{2}.
\]
We put $4(\al-a)-1=4k+3$ and $4(\be-b)-1=4l+3$, where $k,l\geq 0$ so that the
above equation becomes
\[
3000 (2^k5^lx_1)^4 + 200(2^k5^lx_1)^2+1 = {z_2}^{2}.
\]
The elliptic curve defined by $3000X^4+200X^2+1=Y^2$ is isomorphic to the elliptic curve with Cremona label 134400ed1 which has rank zero and torsion subgroup isomorphic to $\Z/2\Z$, hence $(X,Y)=(0,\pm 1)$ are its only affine rational point. Clearly, this point does not provide us with an acceptable pair $(x_1,z_2)$.

Thus we conclude that equation \eqref{eq initial n=2} has no solutions, hence we have
proved the following:
\begin{proposition} \label{prop n=2}
Equation \eqref{eq general} with $n=2$ is impossible.
\end{proposition}

\subsection{The case $n=3$}\label{subsec n=3}

In this case equation \eqref{eq general} becomes
\begin{equation}\label{eq initial n=3}
(x - d)^5 + x^5 + (x + d)^5 = z^3, \quad d = 2^a 5^b,\; a,b\geq 0.
\end{equation}
In accordance with Section \ref{sec:primitive_solutions}
we examine each case (I) through (IV) separately.

Case (I): Elementary symbolic computations transform  \eqref{mult 1-1} 
into the equivalent equation
\[
3\left(\frac{z_1^6}{2^{2(a-\al)}5^{2(b-\be)}}\right)^2+
20\left(\frac{z_1^6}{2^{2(a-\al)}5^{2(b-\be)}}\right) + 10
=\frac{1}{2^{a-\al}5^{b-\be}}\left(\frac{z_2}{2^{a-\al}5^{b-\be}}\right)^3.
\]
We rewrite this as
\begin{equation} \label{eq EC I}
3Y^2+20Y+10=\frac{1}{2^{a-\alpha}5^{b-\beta}}X^3, 
\quad
Y=\frac{z_1^6}{2^{2(a-\alpha)}5^{2(b-\beta)}},\quad
X=\frac{z_2}{2^{a-\alpha}5^{b-\beta}}.
\end{equation}

Case (II): Elementary symbolic computations transform  \eqref{mult 2-1} 
into the equivalent equation
\[
3\left(\frac{5^{2(b-\be)}z_1^6}{2^{2(a-\al)}}\right)^2+
20\left(\frac{5^{2(b-\be)}z_1^6}{2^{2(a-\al)}}\right) + 10
=\frac{5}{2^{a-\al}}\left(\frac{z_2}{2^{a-\al}}\right)^3.
\]
We rewrite this as
\begin{equation} \label{eq EC II}
3Y^2+20Y+10=\frac{5}{2^{a-\alpha}}X^3, 
\quad
Y=\frac{5^{2(\beta-b)}}{2^{2(a-\alpha)}}z_1^6,\quad
X=\frac{z_2}{2^{a-\alpha}}.
\end{equation}

Case (III): Elementary symbolic computations transform  \eqref{mult 3-1} 
into the equivalent equation
\[
3\left(\frac{2^{2(a-\al)}z_1^6}{5^{2(b-\be)}}\right)^2+
20\left(\frac{2^{2(a-\al)}z_1^6}{5^{2(b-\be)}}\right) + 10
=\frac{2}{5^{b-\be}}\left(\frac{z_2}{5^{b-\be}}\right)^3,
\]
which we rewrite as
\begin{equation} \label{eq EC III}
3Y^2+20Y+10=\frac{2}{5^{b-\beta}}X^3, 
\quad
Y=\frac{2^{2(\alpha-a)}}{5^{2(b-\beta)}}z_1^6,\quad
X=\frac{z_2}{5^{b-\beta}}.
\end{equation}

Case (IV): Elementary symbolic computations transform  \eqref{mult 4-1} 
into the equivalent equation
\[
3(2^{2(\al-a)}5^{2(\be-b)}z_1^6)^2 +20(2^{2(\al-a)}5^{2(\be-b)}z_1^6)+10
= z_2^3,
\]
which we rewrite as
\begin{equation} \label{eq EC IV}
3Y^2+20Y+10=10X^3, 
\quad
Y=2^{2(\alpha-a)}5^{2(\beta-b)}z_1^6, \quad X=z_2.
\end{equation}

\vspace*{3pt}
We note that in every equation \eqref{eq EC I}-\eqref{eq EC IV},
the change of variable $X=\mu X_1$, where $\mu$ is an appropriate 
explicit $S$-integer with $S\subseteq \{2,5\}$, depending on the classes 
$\bmod\,3$ of $a-\alpha, b-\beta$, leads to an equation
\begin{equation}  \label{eq uniform expression of EC's}
3Y^2+20Y+10 = cX_1^3,
\end{equation}
where $c$ runs through a ``small'' explicit set of $S$-integers with
$S\subseteq\{2,5\}$ (see below for each case separately).
The elliptic curve defined by \eqref{eq uniform expression of EC's} is
isomorphic to the elliptic curve
\begin{equation} \label{eq common short Weierstrass form}
y_1^2=x_1^3+630c^2, \quad x_1=3cX_1,\; y_1=3c(3Y+10).
\end{equation}

\begin{remark} \label{rem:S_integral_points}
In all cases (I)-(IV), $X$ and $Y$ are $S$-integers with
$S\subseteq\{2,5\}$ and we will see that $X_1$ is also an $S$-integer.
Therefore we will have to compute all $S$-integral points $(x_1,y_1)$ on the 
elliptic curve defined by \eqref{eq common short Weierstrass form}.
For certain values of $c$ the rank of the corresponding elliptic curve is
zero with trivial torsion subgroup, therefore no rational points exist.
For all the remaining values of $c$ the rank is $2$ and we compute the 
$S$-integral points with the aid of the {\sc magma} \cite{Magma} routines  
{\tt SIntegralPoints}  (when $S\neq \emptyset $) or {\tt IntegralPoints}
(when $S=\emptyset$); for the background of these routines we refer, 
respectively, to \cite{PethoZimmerGebelHerrmann99} and 
\cite{StroekerTzanakis94},  \cite{GebelPethoZimmer94} (see also 
\cite{Tzanakisbook13}). We observe that the $y$-coordinate of an $S$-integral point corresponds to a $Y$ that must be equal to the product of an $S$-integer times a sixth power of a non-zero integer. If that does not happen then no solutions of \eqref{eq initial n=3} arise from that point. We ask the reader to have these remarks in mind whenever we expose the 
solutions of \eqref{eq common short Weierstrass form} for the various values 
of $c$.
\end{remark}

We consider equation \eqref{eq common short Weierstrass form} separately
for each case (I)-(IV).

Case (I): We put $(a-\alpha,b-\beta)=(3a_1+i,3b_1+j)$, where 
$0\leq i,j\leq 2$. Then, in \eqref{eq EC I} we have
\[
Y=\frac{z_1^6}{2^{6a_1+2i}5^{2b_1+2j}}\raisepunct{,} \quad
X=\frac{z_2}{2^{3a_1+i}5^{3b_1+j}}\raisepunct{,} \quad
\frac{1}{2^{a-\al}5^{b-\be}}X^3=cX_1^3\,,
\]
where
\[
c=\frac{1}{2^i5^j}\raisepunct{,} 
\quad X_1=\frac{z_2}{2^{4a_1+i}5^{4b_1+j}}\raisepunct{.}
\]
Now we consider equation \eqref{eq common short Weierstrass form} with
\[
c\in\{1,\, 1/2,\, 1/4,\, 1/5,\, 1/25,\, 1/10,\, 1/50,\, 1/20,\, 1/100\}.
\]
When $c\in\{1,\,1/2,\,1/5,\, 1/10,\, 1/50,\, 1/20 \}$, the curve 
defined by the equation \eqref{eq common short Weierstrass form}
is of zero rank with trivial torsion subgroup, hence there are no rational points.

For the remaining values $c=1/4,1/25,1/100$ equation 
\eqref{eq common short Weierstrass form} defines an elliptic curve of 
rank $2$ and we have to compute all $S$-integral points on it, where
$S=\{2,5\}$.
\newline{\indent}
When $c=1/4$, equation \eqref{eq common short Weierstrass form} becomes
$y_1^2=x_1^3+\frac{315}{8}$ and its $S$-integral points are 
\[
(x_1,y_1)=(-\frac{3}{2}\raisepunct{,}\,\pm 6),\:
(-\frac{33}{50}\raisepunct{,}\,\pm\frac{1563}{250})\raisepunct{,}\:
(\frac{9}{4}\raisepunct{,}\,\pm\frac{57}{8})\raisepunct{,}\:
(\frac{849}{256}\raisepunct{,}\,\pm\frac{35673}{4096})\raisepunct{,}\:
(\frac{23}{2}\raisepunct{,}\,\pm\frac{79}{2})\raisepunct{.}
\]
From \eqref{eq common short Weierstrass form}, the values of $Y$ 
corresponding to the above solutions are respectively:
\[
Y=-\frac{2}{3}\raisepunct{,}\;-6,\;-\frac{2^4\,13}{3\cdot 5^3}\raisepunct{,}
\;-\frac{2^2\,191}{5^3}\raisepunct{,}\;
-\frac{1}{6}\raisepunct{,}\;-\frac{13}{2}\raisepunct{,}\; 
\frac{13\cdot 127}{2^{10}\,3}\raisepunct{,}\; 
-\frac{3\cdot 2459}{2^{10}}\raisepunct{,}\;
\frac{2^7}{3^2}\raisepunct{,}\; -\frac{2^2\,47}{3^2}\raisepunct{.}
\] 
Then from Remark \ref{rem:S_integral_points} these values do not lead to a solution and hence for $c=1/4$ equation \eqref{eq EC I} has no solutions.

When $c=1/25$, equation \eqref{eq common short Weierstrass form},
becomes $y_1^2=x_1^3+\frac{126}{125}$ and its $S$-integral points are
\[
(x_1,y_1)=(-\frac{1}{5}\raisepunct{,}\,\pm 1),\; 
(\frac{1009}{2500}\raisepunct{,}\,\pm\frac{129527}{125000})\raisepunct{,}\;
(\frac{69}{80}\raisepunct{,}\,\pm\frac{411}{320})\raisepunct{,}\;
(\frac{99}{25}\raisepunct{,}\,\pm\frac{993}{125}),
\]
with corresponding values of $Y$
\[
Y=-\frac{5}{3^2}\raisepunct{,}\; -\frac{5\cdot 11}{3^2}\raisepunct{,}\;
-\frac{59\cdot 347}{2^3\,3^2\,5^4}\raisepunct{,}\;
-\frac{31\cdot 71\cdot 127}{2^3\,3^2\,5^4}\raisepunct{,}\;
\frac{15}{2^6}\raisepunct{,}\; -\frac{5^2\,53}{2^6\,3}\raisepunct{,}\;
\frac{281}{15}\raisepunct{,}\; -\frac{127}{5}\raisepunct{.}
\]
Again by remark \ref{rem:S_integral_points} these values do not lead to a solution.

Finally --for case (I)-- if $c=1/100$ then 
\eqref{eq common short Weierstrass form} becomes 
$y_1^2=x_1^3+\frac{63}{1000}$ and all its $S$-integral points are
\begin{align*}
(x_1,y_1)= & (-\frac{159}{400}\raisepunct{,}\,
                      \pm\frac{111}{8000})\raisepunct{,}\;
(\frac{1}{100}\raisepunct{,}\,\pm\frac{251}{1000})\raisepunct{,}\;
(\frac{3}{10}\raisepunct{,}\,\pm\frac{3}{10})\raisepunct{,}\;
(\frac{81}{100}\raisepunct{,}\,\pm\frac{771}{1000})\raisepunct{,}
\\[2pt]
 & (\frac{129921}{16\cdot 10^4}\raisepunct{,}\,
 \pm\frac{49508031}{64\cdot 10^6})\raisepunct{,}\;
 (\frac{33}{10},\,\pm 6).
 \end{align*}
These respectively give
\begin{align*}
Y= & -\frac{7\cdot 109}{2^4\,15}\raisepunct{,}\; 
-\frac{3^2\,31}{2^4\,5}\raisepunct{,}\;
  -\frac{7^2}{90}\raisepunct{,}\; 
  -\frac{19\cdot 29}{90}\raisepunct{,}\;
  0,\; -\frac{2^2\,5}{3}\raisepunct{,}\; 
  \frac{157}{30}\raisepunct{,}\; -\frac{119}{10}\raisepunct{,}\;
  \frac{13\cdot 83\cdot 3121}{2^{10}\,5^4}\raisepunct{,}\;
  \\[2pt]
 & -\frac{7\cdot 47\cdot 67\cdot 1039}{2^{10}\,5^4\,3}\raisepunct{,}\;
  \frac{190}{3}\raisepunct{,}\; -70.
\end{align*}
As above these values do not lead to a solution.

\emph{Conclusion: No solutions to \eqref{eq EC I} exist.}

Case (II): We put $a-\alpha=3a_1+i $ with $0\leq i\leq 2$. 
Then, in \eqref{eq EC II} we have
\[
Y=\frac{5^{\be-b}z_1^6}{2^{6a_1+2i}}\raisepunct{,} \quad
X=\frac{z_2}{2^{3a_1+i}}\raisepunct{,} \quad
\frac{5}{2^{a-\al}}X^3=cX_1^3\,,
\]
where
\[
c=\frac{5}{2^i}\raisepunct{,} 
\quad X_1=\frac{z_2}{2^{4a_1+i}}\raisepunct{.}
\]
Thus we consider equation \eqref{eq common short Weierstrass form} with
$c\in\{5,\,5/2,\,5/4\}$. Now $X,Y$ and $X_1$ are $S$-integers with 
$S=\{2\}$, and $(x_1,y_1)$ is $S$-integral solution to 
\eqref{eq common short Weierstrass form}. 

If $c=5$, then \eqref{eq common short Weierstrass form} becomes 
$y_1^2=x_1^3+15750$ and all its $S$-integral solutions are
\[
(x_1,y_1)=(-5,\,\pm 125),\;
(\frac{345}{16}\raisepunct{,}\,\pm\frac{10275}{64}),
\;(99,\,\pm 993).
\]
These furnish us with the following values of $Y$:
\[
Y=-\frac{5}{9}\raisepunct{,}\;-\frac{55}{9}\raisepunct{,}\;
\frac{15}{2^6}\raisepunct{,}\; \frac{5^2\,53}{2^6\,3}\raisepunct{,}\; 
\frac{281}{15}\raisepunct{,}\; -\frac{127}{5}\raisepunct{.}
\]
Again by Remark \ref{rem:S_integral_points} the above values do not lead to a solution. Therefore, when $c=5$, \eqref{eq EC II} has no solutions.

When $c=5/2$, the elliptic curve defined by 
\eqref{eq common short Weierstrass form} is of zero rank with trivial
torsion subgroup, hence there are no rational solutions.

When $c=5/4$, \eqref{eq common short Weierstrass form} becomes 
$y_1^2=x_1^3+7875/8$ and its $S$-integral solutions are
\[
(x_1,y_1)= (-\frac{159}{16}\raisepunct{,}\,\pm\frac{111}{64})\raisepunct{,}\;
(\frac{1}{4}\raisepunct{,}\,\pm\frac{251}{8})\raisepunct{,}\;
(\frac{15}{2}\raisepunct{,}\,\pm\frac{75}{2})\raisepunct{,}\;
(\frac{81}{4}\raisepunct{,}\,\pm\frac{771}{8})\raisepunct{,}\;
(\frac{165}{2}\raisepunct{,}\,\pm 750),
\]
with the corresponding values of $Y$ being
\begin{align*}
 Y= & -\frac{7\cdot 109}{2^4\cdot 15}\raisepunct{,},\;
 -\frac{3^2\cdot 31}{2^4\cdot 5}\raisepunct{,}\;
-\frac{7^2}{90}\raisepunct{,}\; 
-\frac{19\cdot 29}{90}\raisepunct{,}\;
0,\; -\frac{2^2\,5}{3}\raisepunct{,}\; 
\frac{157}{30}\raisepunct{,}\; -\frac{7\cdot 17}{10}\raisepunct{,}\; 
-\frac{7\cdot 17}{10}\raisepunct{,} \\
& \frac{10\cdot 19}{3}\raisepunct{,}\; -70.   
\end{align*}
Similar to the above case these values of $Y$ do not lead to a solution of \eqref{eq EC II}.

\emph{Conclusion: No solutions to \eqref{eq EC II} exist.}

Case (III): We put $b-\beta=3b_1+j $ with $0\leq j\leq 2$. 
Then, in \eqref{eq EC III} we have
\[
Y=\frac{2^{\al-a}z_1^6}{5^{6b_1+2j}}\raisepunct{,} \quad
X=\frac{z_2}{5^{3b_1+j}}\raisepunct{,} \quad
\frac{2}{5^{b-\be}}X^3=cX_1^3\,,
\]
where
\[
c=\frac{2}{5^j}\raisepunct{,} 
\quad X_1=\frac{z_2}{2^{4b_1+j}}\raisepunct{.}
\]
Thus we consider equation \eqref{eq common short Weierstrass form} with
$c\in\{2,\,2/5,\,2/25\}$. Now $X,Y$ and $X_1$ are $S$-integers with $S=\{5\}$, 
and $(x_1,y_1)$ is $S$-integral solution to 
\eqref{eq common short Weierstrass form}. 

If $c=2$, then \eqref{eq common short Weierstrass form} becomes 
$y_1^2=x_1^3+2520$. Its $S$-integer solutions are
\[
(x_1,y_1)=(-6,\,\pm 48),\;
(-\frac{66}{25}\raisepunct{,}\,\pm\frac{6252}{125})\raisepunct{,}\,\;
(9,\,\pm 57),\; (46,\,\pm 316),
\]
and the corresponding values of $Y$ are
\[
Y=-\frac{2}{3}\raisepunct{,}\; -6,\; -\frac{2^4\,13}{3\cdot 5^3}\raisepunct{,}\;
-\frac{2^2\,191}{5^3}\raisepunct{,}\; -\frac{1}{6}\raisepunct{,},\;
-\frac{13}{2}\raisepunct{,}\;
\frac{2^7}{3^2}\raisepunct{,}\; -\frac{2^2\,47}{3^2},
\]
and none of these values of $Y$ lead to a solution of \eqref{eq EC III}.

When $c=2/5$, the elliptic curve defined by 
\eqref{eq common short Weierstrass form} is of zero rank with trivial
torsion subgroup, hence there are no rational solutions.

When $c=2/25$, \eqref{eq common short Weierstrass form} becomes 
$y_1^2=x_1^3+504/125$, the $S$-integral solutions of which are
\[
(x_1,y_1)=(\frac{1}{25}\raisepunct{,}\,\pm\frac{251}{125})\raisepunct{,}\;
(\frac{6}{5}\raisepunct{,}\,\pm\frac{12}{5})\raisepunct{,}\;
(\frac{81}{25}\raisepunct{,}\,\pm\frac{771}{125})\raisepunct{,}\;
(\frac{66}{5}\raisepunct{,}\,\pm 48),
\]
with the corresponding values of $Y$ being
\[
Y=-\frac{7^2}{90}\raisepunct{,}\; 
-\frac{19\cdot 29}{90}\raisepunct{,}\;
0,\; -\frac{2^2\,5}{3}\raisepunct{,}\; 
\frac{157}{30}\raisepunct{,}\; -\frac{7\cdot 17}{10}\raisepunct{,}\;
\frac{10\cdot 19}{3}\raisepunct{,}\; -70.
\]
However, no non-zero value of $Y$ lead to a solution of \eqref{eq EC III}.

\emph{Conclusion: No solutions to \eqref{eq EC III} exist.}

Case (IV): Now, $X,Y$ in equation \eqref{eq EC IV} are integers and this
equation is equivalent to $y_1^2=x_1^3+63000$, where $x_1=30X$ and 
$y_1=30(3Y+10)$. All integer solutions of this equation (if we forget the
above special form of $x_1,y_1$) are 
\[
(x_1,y_1)=(1,\,\pm 251),\; (30,\,\pm 300),\;(81,\pm 771),\;
(330,\,\pm 6000),
\]
from which only the solution $(x_1,y_1)=(330,-6000)$ returns to
a non-zero integral solution $(X,Y)$, namely, $(X,Y)=(11,-70)$. But this value
of $Y$ is not of the form required by \eqref{eq EC IV}.

\emph{Conclusion: No solutions to \eqref{eq EC IV} exist.} 

In view of our previous conclusions we have proved that equation \eqref{eq initial n=3}
has no solutions which proves the following:
\begin{proposition} \label{prop p=3}
Equation \eqref{eq general} with $n=3$ is impossible.
\end{proposition}

\subsection{The case $n= 5$}  \label{subsec n=5}

Suppose $(x,z)$ is a solution of \eqref{eq general} for $n=5$. 
Let $r = (x,d)$, then $x = x_1r$ and $d = d_1r$ with 
$(x_1,d_1)=1$, hence $r^5\mid z^5$ and consequently $r\mid z$. Setting
$z=rz_1$ we obtain the equation
\[
(x_1 - d_1)^5 + x_1^5 + (x_1 + d_1)^5 = z_1^5, \quad (x_1,d_1)=1.
\]
By Theorem 1.1 of \cite{BennettKoutsianas20} this equation has nonzero 
integer solutions only when $d_1=2$, in which case the only solution is 
$(x_1,z_1)=\pm(1,3)$.
It follows that $d=2r$, which shows that, in $d=2^a5^b$ we have $a\geq 1$,
hence $r=2^{a-1}5^b$ and 
$(x,z)=\pm(2^{a-1}5^b,3\cdot2^{a-1}5^b)=\pm (d/2,3d/2)$.
Thus we have proved the following:
\begin{proposition}\label{prop:solutions_n_5}
Equation \eqref{eq general} with $n=5$ has integer solutions only if $a\geq 1$,
in which case all its integer solutions are given by
$(x,z)=\pm(d/2,3d/2)$.
\end{proposition}

%%%%%%%%%%%%%%%%%%%%%%%%%%%%%%%%%%%%%%%%%%%%%%%%%
\section{The modular method for $n\geq 7$} \label{sec:Frey_curves}

As already noted in equation \eqref{eq general}, it suffices
to consider only prime values of $n$.
In this section we prove that there are no solutions of 
\eqref{eq general} when $n$ is a prime 
greater or equal to $7$. In the proof we make use of  the 
modular method which has its origin in the proof of Fermat's Last 
Theorem \cite{Wiles95}. The main idea in the modular method is to 
attach Frey-Hellegourach curves associated with our equations and 
using the modularity of elliptic curves 
\cite{Wiles95,TaylorWiles95,BreuilConradDiamondTaylor01}, the work of 
Mazur \cite{Mazur78} and Ribet's level-lowering theorem 
\cite{Ribet90}, to compare Galois representations. 
For the rest of the section we assume that $n\geq 7$ is a prime.

Before we study the equations \eqref{mult 1-1}-\eqref{mult 4-2} with 
the modular method we have to recall some standard results and 
terminology. The reader can find a more detailed exposition of the 
techniques and ideas in, for example, \cite[Chapter 15]{Cohen07}.

Suppose $f$ is a cuspidal newform of weight $2$ and 
level $N$ with $q$-expansion
$$
f = q + \sum_{i=2}^\infty a_i(f)q^i.
$$
We denote by $K_f$ the \textit{eigenvalue field} of $f$ and say that 
$f$ is \textit{irrational} if $[K_f:\Q]>1$ and \textit{rational} 
otherwise. Suppose $n$ is a rational prime and $\fn\mid n$ a prime 
ideal in $K_f$ above $n$. Then, we can associate a continuous, 
semisimple Galois representation
$$
\brhof:\Gal(\bar{\Q}/\Q)\rightarrow \GL_2(\bar{\mathbb{F}}_n),
$$
that is unramified at all primes $\ell\nmid nN_f$ and 
$\Tr(\brhof(\Frob_\ell))\equiv a_\ell(f)\pmod{\fn}$ where $\Frob_\ell$ 
is a Frobenius element at $\ell$.

Suppose $E$ is an elliptic curve over $\Q$ with conductor $N_E$. 
For a prime $\ell$ of good reduction of $E$, we let 
$a_\ell(E) = \ell + 1 - \#\tilde{E}(\Fl)$, where $\tilde{E}$ is the 
reduction of $E$ at $\ell$. We denote by $\brhoE$ the Galois 
representation of $\Gal(\bar{\Q}/\Q)$ acting on the $n$-torsion 
subgroup of $E$.

The following proposition provides a 
standard technique that is used to bound $n$ and its origin goes back 
to Serre \cite{Serre87}. 
\begin{proposition}\label{prop:elimination_difference}
Suppose $f$ is a cuspidal newform of weight $2$, level $N_f$ and 
trivial character with eigenvalue field $K_f$. We assume that 
$\brhof\simeq\brhoE$ where $\brhof$ is the associate to $f$ residual 
representation of $\Gal(\bar{\Q}/\Q)$ and $\fn\mid n$ is a prime ideal 
in $K_f$. Let $\ell\neq n$ be a prime, then
\begin{itemize}
    \item if $\ell\nmid N_E N_f$ then $a_\ell(f)\equiv 
    a_\ell(E)\pmod{\fn}$,
    \item if $\ell\nmid N_f$ and $\ell\|N_E$ 
              then $a_\ell(f)\equiv\pm (\ell + 1)\pmod{\fn}$,
\end{itemize}
where $a_\ell(f)$ is the Hecke eigenvalue of $f$ at $\ell$.
\end{proposition}

%%%%%%%%
\subsection{Frey-Hellegourach curves of signature $(n,n,2)$}

We apply the recipes of Bennett and Skinner \cite[Section 2]{BennettSkinner04} to equations \eqref{mult 1-1} - \eqref{mult 4-2}. According to the different values of $a-\alpha$ and $b-\beta$ we attach the corresponding Frey-Hellegourach curve.

\paragraph{\bf Case I} We consider the equations 
\eqref{mult 1-1} and \eqref{mult 1-2}. We recall that 
$a\geq \alpha$ and $b\geq \beta$. For any value of $a-\alpha$ 
the equation \eqref{mult 1-1} satisfies the conditions of case 
(ii) in \cite[Section 2]{BennettSkinner04}.

Next we focus on equation \eqref{mult 1-2}. Suppose that 
$a\geq \alpha + 2$, then we are in case (v) in \cite[Section 2]{BennettSkinner04}. When $a=\alpha$ then we are in case (ii) and when  $a=\alpha + 1$ then we are in case (iv) in \cite[Section 2]{BennettSkinner04}. The Frey-Hellegouarch curves are represented in the Table \ref{tab:Frey_curves_I}.

\begin{table}
    \centering
    \begin{tabular}{||c |c| c||} 
        \hline\hline
        Case & Equation  &  Frey-Hellegourach Curve\\  [0.5ex] 
        \hline\hline
        
        $a \geq \alpha + 2$ & \eqref{mult 1-1} & $E_{I,1}: Y^2=X^3+ 20 ({z_1}^{2n}+2^{2a-2\alpha}5^{2b-2\beta})X^2 + 70{z_1}^{4n}X$ \\[1ex] 
        & \eqref{mult 1-2} &   $F_{I,1}: Y^2 +XY= X^3 - \frac{3z_1^{2n}+2^{2a-2\alpha+1}\cdot 5^{2b-2\beta+1}+1}{4} X^2$\\[0.5ex]
        & & $+ 7\cdot 2^{4(a - \alpha)-5}\cdot 5^{4(b - \beta)+1}X$ \\ [1ex]
        \hline

        $a = \alpha$ & \eqref{mult 1-1} & $E_{I,2}: Y^2=X^3+ 20 ({z_1}^{2n}+5^{2b-2\beta})X^2 + 70{z_1}^{4n}X$ \\[1ex] 
        & \eqref{mult 1-2} &   $F_{I,2}: Y^2 = X^3 + 2 (3{z_1}^{2n}+2\cdot 5^{2b-2\beta+1}) X^2+ 14\cdot 5^{4(b-\beta)+1}X$ \\ [1ex]
        \hline
        
        $a = \alpha + 1$ & \eqref{mult 1-1} & $E_{I,3}: Y^2=X^3+ 20 ({z_1}^{2n}+2^2. 5^{2b-2\beta})X^2 + 70{z_1}^{4n}X$ \\[1ex]
        & \eqref{mult 1-2} &   $F_{I,3}: Y^2 = X^3 - (3{z_1}^{2n}+2^3\cdot 5^{2b-2\beta+1}) X^2+ 56\cdot 5^{4(b-\beta)+1}X$ \\
        \hline\hline
    \end{tabular}
    \caption{Frey-Hellegourach curves for case I. It holds $a\geq \alpha$ and $\beta\leq b$.}\label{tab:Frey_curves_I}
\end{table}
\paragraph{\bf Case II} We consider the equations 
\eqref{mult 2-1} and \eqref{mult 2-2}. We recall that 
$a\geq \alpha$ and $\beta> b$. For any value of $a-\alpha$ the 
equation \eqref{mult 2-1} satisfies the conditions of case (ii) 
in \cite[Section 2]{BennettSkinner04}.

Next we turn to equation \eqref{mult 2-2}. Suppose that 
$a\geq \alpha + 2$, then we are in case (v) in 
\cite[Section 2]{BennettSkinner04}. When $a=\alpha$ then we are 
in case (ii) and when  $a=\alpha + 1$ then we are in case (iv) 
of \cite[Section 2]{BennettSkinner04}. The Frey-Hellegouarch 
curves are represented in the Table \ref{tab:Frey_curves_II}.
\begin{table}
    \centering
    \begin{tabular}{||c |c| c||} 
        \hline
        Case & Equation  &  Frey-Hellegourach Curve\\  [0.5ex] 
        \hline\hline
        $a\geq \alpha + 2$  & \eqref{mult 2-1}  & $E_{II,1}: Y^2= X^3 + 4(5^{2\beta-2b}\cdot z_1^{2n} + 2^{2a-2\alpha})X^2+ 14\cdot 5^{4(\beta-b)-1}\cdot z_1^{4n}X$\\[1ex]
        & \eqref{mult 2-2}& $F_{II,1}: Y^2 +XY= X^3 -  \frac{3\cdot 5^{2\beta-2b}z_1^{2n} + 5\cdot 2^{2a-2\alpha +1}+1}{4} X^2 + 35\cdot 2^{4(a - \alpha)-5} X$\\[1ex]
        \hline
 
        $a=\alpha$  & \eqref{mult 2-1}  & $E_{II,2} : Y^2= X^3 + 4(5^{2\beta-2b}\cdot z_1^{2n} + 1)X^2+ 14\cdot 5^{4(\beta-b)-1}\cdot z_1^{4n}X$\\[1ex]
        & \eqref{mult 2-2}& $F_{II,2} : Y^2 = X^3 +10 (3\cdot 5^{2\beta-2b-1}z_1^{2n} +2) X^2 + 70 X$\\[1ex]
        \hline
 
        $a = \alpha + 1$  & \eqref{mult 2-1}  & $E_{II,3} : Y^2= X^3 + 4(5^{2\beta-2b}\cdot z_1^{2n} + 2^2)X^2+ 14\cdot 5^{4(\beta-b)-1}\cdot z_1^{4n}X$\\[1ex]
        & \eqref{mult 2-2}& $F_{II,3}: Y^2 = X^3 - (3\cdot 5^{2\beta-2b}z_1^{2n} +40) X^2 + 280X$\\
        \hline
    \end{tabular}
    \caption{Frey-Hellegourach curves for the case II. It holds $a\geq \alpha$ and $\beta> b$.}\label{tab:Frey_curves_II}
\end{table}
\paragraph{\bf Case III} We consider the equations 
\eqref{mult 3-1} and \eqref{mult 3-2}. We recall that 
$\alpha\geq a+1$ and $b\geq \beta$. For any value of $a-\alpha$ 
equation \eqref{mult 3-2} satisfies the conditions of case (ii) 
of \cite[Section 2]{BennettSkinner04}.

Now we focus on equation \eqref{mult 3-1}. If  
$\alpha\geq a + 2$, then we are in case (v) of 
\cite[Section 2]{BennettSkinner04}. When $\alpha = a+1$, we are 
in case (iv) of \cite[Section 2]{BennettSkinner04}. The 
Frey-Hellegouarch curves are represented in the Table 
\ref{tab:Frey_curves_III}.
\begin{table}
    \centering
    \begin{tabular}{||c |c| c||} 
        \hline
        Case & Equation  &  Frey-Hellegourach Curve\\  [0.5ex] 
        \hline\hline
 
        $\alpha\geq (a+2)$ & \eqref{mult 3-1}& $E_{III,1}: Y^2+ XY= X^3+ \frac{5(2^{2\alpha-2a}z_1^{2n}+5^{2b-2\beta})-1}{4}X^2$\\[0.5ex]
        & & $+35\cdot 2^{4(a - \alpha)-7}z_1^{4n} X $\\[1ex]
        & \eqref{mult 3-2}& $F_{III,1} : Y^2= X^3+ 4(3\cdot 2^{2\alpha-2a-1}z_1^{2n}+5^{2b-2\beta+1})X^2+14\cdot 5^{4(b - \beta)+1} X $\\[1ex]
        \hline

        $\alpha=a+1$  & \eqref{mult 3-1} & $E_{III,2}: Y^2= X^3+ 5(2^{2}z_1^{2n}+5^{2b-2\beta})X^2+70\cdot z_1^{4n}X $\\[1ex] 
        & \eqref{mult 3-2} & $F_{III,2} : Y^2= X^3+ 4(6z_1^{2n}+5^{2b-2\beta+1})X^2+14\cdot 5^{4(b - \beta)+1} X $\\  
        \hline
    \end{tabular}
    \caption{Frey-Hellegourach curves for the case III. It holds $\alpha\geq a+1$ and $b\geq \beta$.}\label{tab:Frey_curves_III}
\end{table}
\paragraph{\bf Case IV} We consider the equations 
\eqref{mult 4-1} and \eqref{mult 4-2}. We recall that 
$\alpha\geq a+1$ and $\beta\geq b+1$. For any value of 
$a-\alpha$ the equation \eqref{mult 4-2} satisfies the 
conditions of case (ii) in \cite[Section 2]{BennettSkinner04}.

We turn to equation \eqref{mult 4-1} now. If 
$\alpha\geq a + 2$, then we are in case (v) of 
\cite[Section 2]{BennettSkinner04}. When $\alpha = a+1$, 
we are in case (iv) of \cite[Section 2]{BennettSkinner04}. 
The Frey-Hellegouarch curves are represented in the Table 
\ref{tab:Frey_curves_IV}.
\begin{table}
    \centering
    \begin{tabular}{||c |c| c||} 
        \hline
        Case & Equation  &  Frey-Hellegourach Curve\\  [0.5ex] 
        \hline\hline

        $\alpha\geq (a+2)$  & \eqref{mult 4-1} & $E_{IV,1}: Y^2 +XY = X^3 + 2^{2\alpha-2a-2}5^{2\beta-2b}z_1^{2n}X^2$\\
        & & $+ 7\cdot 2^{4(\alpha-a)-7} 5^{4(\beta-b)-1} z_1^{4n}X$\\[1ex]
        & \eqref{mult 4-2} & $F_{IV,1}: Y^2= X^3 +20(3\cdot 2^{2\alpha-2a-1}5^{2\beta - 2b -1}z_1^{2n} + 1)X^2+ 70X$\\[1ex]
        \hline
        
        $\alpha=a+1$  & \eqref{mult 4-1} & $E_{IV,2}: Y^2 = X^3 + (4\cdot 5^{2\beta-2b}z_1^{2n}+1)X^2 + 14\cdot 5^{4(\beta-b)-1} z_1^{4n}X$\\[1ex]
        & \eqref{mult 4-2} & $F_{IV,2}: Y^2= X^3 +20(6\cdot 5^{2\beta - 2b -1} z_1^{2n} + 1)X^2+ 70X$\\[1ex]
        \hline
    \end{tabular}
    \caption{Frey-Hellegourach curves for the case IV. It holds $\alpha\geq a+1$ and $\beta\geq b+1$.}\label{tab:Frey_curves_IV}
\end{table}

Let $E=E_{i,k}$ or $F_{i,k}$ as above. We denote by $\brhoE$ the Galois representation of $\Gal(\bar{\Q}/\Q)$ acting on the $n$-torsion points of $E$.

\begin{proposition}\label{prop:irreducibility}
The representation $\brhoE$ is absolutely irreducible.
\end{proposition}

\begin{proof}
This is an immediate consequence of \cite[Corollary 3.1]{BennettSkinner04}, based on work by Mazur \cite{Mazur78}, and the fact that $z_1z_2\neq \pm 1$.
\end{proof}

Let 
$$
N_n(E) =N(E)\Bigg/\prod_{q\mid z_1z_2}q.
$$

%We denote by $N_n(E)$ the (Serre) level of $\brhoE$. 

\begin{proposition}\label{prop:level_lowering}
Suppose $\brhoE$ is as above. Then there exists a newform $f$ of trivial character, weight $2$ and level $N_n(E)$ and a prime ideal $\fn\mid n$ of $K_f$ such that 
$$
\brhoE\simeq \brhof.
$$
\end{proposition}

\begin{proof}
This is an immediate consequence of modularity of elliptic curves \cite{Wiles95,TaylorWiles95,BreuilConradDiamondTaylor01}, Proposition \ref{prop:irreducibility}, Table \ref{tab:Disc_Conductor} and Ribet's level lowering \cite{Ribet90}.
%
%From the modularity of elliptic curves \cite{Wiles95,TaylorWiles95,BreuilConradDiamondTaylor01} we know that $\brhoE$ is modular. From Proposition \ref{prop:irreducibility} we also know that $\brhoE$ is absolutely irreducible. On the other hand, from Table \ref{tab:Disc_Conductor} and $p\nmid 2\cdot 3\cdot 5\cdot 7$ the valuation of the discriminant of $E$ is divisible by $n$ and $E$ is semistable at $p$, hence $\brhoE$ is a finite flat at $p$. From all the above and Ribet's level lowering \cite[Theorem 1.1]{Ribet90} we have the result.  
\end{proof}

In Table \ref{tab:Disc_Conductor} we have computed the $N_n(E)$ of $\brhoE$ according to \cite[Lemmas 2.1 and 3.3]{BennettSkinner04}.

\begin{table}
    \centering
    \begin{tabular}{|| c | c | c ||}
        \hline
        Frey curve & Discriminant $\Delta(E)$ & $N_n(E)$\\
        \hline\hline
        $E_{I,1}$ & $2^9\cdot 5^3\cdot 7^2 \cdot (z_2 z_1^8)^n$ & $2^8\cdot 5^2\cdot 7$\\
        \hline
        $F_{I,1}$ & $2^{8(a-\alpha)-10}\cdot 3\cdot 5^{8(b-\beta)+2}\cdot 7^2 \cdot z_2^n$ & $2\cdot 3\cdot 5\cdot 7$\\
        \hline
        $E_{I,2}$ & $2^9\cdot 5^3\cdot 7^2 \cdot (z_2 z_1^8)^n$ & $2^8\cdot 5^2\cdot 7$\\
        \hline
        $F_{I,2}$ & $2^{8}\cdot 3\cdot 5^{8(b-\beta)+2}\cdot 7^2 \cdot z_2^n$ & $2^7\cdot 3
        \cdot 5\cdot 7$\\
        \hline
        $E_{I,3}$ & $2^9\cdot 5^3\cdot 7^2 \cdot (z_2 z_1^8)^n$ & $2^8\cdot 5^2\cdot 7$ \\
        \hline
        $F_{I,3}$ & $2^{10}\cdot 3\cdot 5^{8(b-\beta)+2}\cdot 7^2 \cdot z_2^n$ & $2^3\cdot 3\cdot 5\cdot 7$\\
        \hline
        $E_{II,1}$ & $2^9\cdot 5^{8(\beta-b)-2}\cdot 7^2 \cdot (z_2 z_1^8)^n$ & $2^8\cdot 5\cdot 7$\\
        \hline
        $F_{II,1}$ & $2^{8(a-\alpha)-10}\cdot 3\cdot 5^{3}\cdot 7^2 \cdot z_2^n$ & $2\cdot 3\cdot 5^2\cdot 7$\\
        \hline
        $E_{II,2}$ & $2^9\cdot 5^{8(\beta-b)-2}\cdot 7^2 (z_2 z_1^8)^n$ & $2^8\cdot 5\cdot 7$\\
        \hline
        $F_{II,2}$ & $2^{8}\cdot 3\cdot 5^{3}\cdot 7^2 \cdot z_2^n$ & $2^7\cdot 3\cdot 5^2\cdot 7$\\
        \hline
        $E_{II,3}$ & $2^9\cdot 5^{8(\beta-b)-2}\cdot 7^2 \cdot (z_2 z_1^8)^n$ & $2^8\cdot 5\cdot 7$\\
        \hline
        $F_{II,3}$ & $2^{10}\cdot 3\cdot 5^{3}\cdot 7^2 \cdot z_2^n$ & $2^3\cdot 3\cdot 5^2\cdot 7$\\
        \hline
        $E_{III,1}$ & $2^{8(\alpha-a)-14}\cdot 5^{3}\cdot 7^2 \cdot (z_2z_1^8)^n$ & $2\cdot 5^2\cdot 7$\\
        \hline
        $F_{III,1}$ & $2^9\cdot 3\cdot 5^{8(b-\beta)+2}\cdot 7^2 \cdot z_2^n$ & $2^8\cdot 3\cdot 5\cdot 7$\\
        \hline
        $E_{III,2}$ & $2^{6}\cdot 5^{3}\cdot 7^2 \cdot (z_2z_1^8)^n$ & $2^5\cdot 5^2\cdot 7$\\
        \hline
        $F_{III,2}$ & $2^9\cdot 3\cdot 5^{8(\beta-b)+2}\cdot 7^2 \cdot z_2^n$ & $2^8\cdot 3\cdot 5\cdot 7$\\
        \hline
        $E_{IV,1}$ & $2^{8(\alpha-a)-14}\cdot 5^{8(\beta-b)-2}\cdot 7^2 \cdot (z_2z_1^8)^n$ & $2\cdot 5\cdot 7$\\
        \hline
        $F_{IV,1}$ & $2^9\cdot 3\cdot 5^3\cdot 7^2 \cdot z_2^n$ & $2^8\cdot 3\cdot 5^2\cdot 7$\\
        \hline
        $E_{IV,2}$ & $2^{6}\cdot 5^{8(\beta-b)-2}\cdot 7^2 \cdot (z_2z_1^8)^n$ & $2^5\cdot 5\cdot 7$\\
        \hline
        $F_{IV,2}$ & $2^9\cdot 3\cdot 5^3\cdot 7^2 \cdot z_2^n$ & $2^8\cdot 3\cdot 5^2\cdot 7$\\
        \hline
    \end{tabular}
    \caption{The discriminant and $N_n(E)$ of the Frey-Hellegouarch curves.}
    \label{tab:Disc_Conductor}
\end{table}

%%%%%%%%%%%%%%%%%%%%%%%%%%%%%%%%%%%%%
\section{Proof of Theorem \ref{Thm1} for $n\geq 7$}\label{sec:Proof large n}
\begin{proof}[Proof of Theorem \ref{Thm1}]
Suppose $(x,z)$ is a solution of the equation \eqref{eq general} for some value of $d$ where $d=2^a5^b$ with $a,b\geq 0$ are integers and $n\geq 7$ is a prime. As we explain in Section \ref{sec:primitive_solutions} there exist integers $z_1,z_2,u_1$ and $u_2$ with $(z_1,z_2)=1$, $(z_1z_2,10)=1$ and $u_1,u_2$ are $\{2,5\}$-units such that 
\begin{align*}
    x & = u_1z_1^n,\\
    P & = u_2z_2^n,
\end{align*}
where $P=3x^4 + 20d^2x^2 + 10d^4$. According to the valuation of $u_1$, $u_2$ and $d$ at $2$ and $5$ we have four possible cases (I)-(IV) and for each case we construct two Fermat type equations of signature $(n,n,2)$ for the pair $(z_1,z_2)$; the equations \eqref{mult 1-1}-\eqref{mult 4-2}. From the work of Bennett and Skinner \cite{BennettSkinner04} we attach for each case and pair $(z_1,z_2)$ two Frey-Hellegouarch curves, $E_{i,k}$ and $F_{i,k}$, as we have explained in Section \ref{sec:Frey_curves} (see Tables \ref{tab:Frey_curves_I}-\ref{tab:Frey_curves_IV}). 

We denote by $\brhoEi$ and $\brhoFi$ the Galois representation of $\Gal(\bar{\Q}/\Q)$ acting on the $n$-torsion points of $E_{i,k}$ and $F_{i,k}$, respectively. We denote by $N_n(E_{i,k})$ and $N_n(F_{i,k})$ the Serre level of $\brhoEi$ and $\brhoFi$, respectively (see Table \ref{tab:Disc_Conductor}). Then, from Proposition \ref{prop:level_lowering} we know that there exist a newform $f$ (resp. $g$) of weight $2$, trivial character and level $N_n(E_{i,k})$ (resp. $N_n(F_{i,k})$) such that $\brhof\simeq\brhoEi$ (resp $\brhoFi\simeq \brhog)$ where $\fn\mid n$ (resp. $\fnp\mid n$) is a prime ideal of $K_f$ (resp. $K_g$). 

Because for each pair $(z_1,z_2)$ we have attached two Frey curves we will apply the powerful multi-Frey approach to bound $n$ \cite{BugeaudLucaMignotteSiksek08,BugeaudMignotteSiksek08}. Suppose $\ell\neq 2,3,5,7$ is a prime. We define $\Da=|a - \alpha|$ and $\Db = |b-\beta|$. We also define
$$
R_\ell(f) = 
\begin{cases}
N_{K_f/\Q}\left(a_\ell(E_{i,k}) - a_\ell(f)\right), & \ell\nmid \Delta(E_{i,k}),\\
N_{K_f/\Q}\left((\ell + 1)^2 - a_\ell^2(f)\right), & \ell\mid \Delta(E_{i,k}),
\end{cases}
$$
where $\Delta(E_{i,k})$ is the discriminant of $E_{i,k}$. Similarly, we define
$$
R_\ell^\prime(g) = 
\begin{cases}
N_{K_g/\Q}\left(a_\ell(F_{i,k}) - a_\ell(g)\right), & \ell\nmid \Delta(F_{i,k}),\\
N_{K_g/\Q}\left((\ell + 1)^2 - a_\ell^2(g)\right), & \ell\mid \Delta(F_{i,k}),
\end{cases}
$$
where $\Delta(F_{i,k})$ is the discriminant of $F_{i,k}$. It is important to mention that both $R_\ell(f)$ and $R_\ell^\prime(g)$ depend on the residue class of $(z_1,z_2)$ modulo $\ell$ and $(\Da,\Db)$ modulo $(\ell - 1)$. Now let
$$
T_\ell(f,g) = \ell\cdot \prod_{\substack{(z_1,z_2)\in\Fl^2,\\(\Da,\Db)\in\left(\mathbb{Z}/(\ell - 1)\mathbb{Z}\right)^2}} \gcd(R_\ell(f),R_\ell^\prime(g)).
$$
From Proposition \ref{prop:elimination_difference} we know that if a solution $(z_1,z_2)$ arises from the pair of newforms $(f,g)$ then it should hold $n\mid T_\ell(f,g)$.

We have written a Magma script that computes $U(f,g)=\gcd_{\ell\leq B}(T_\ell(f,g))$ where $B$ is a suitable positive integer. For the majority of the pairs $(f,g)$ it is enough to consider $B = 19$ to deduce that $n\leq 5$. However, there are a few pairs $(f,g)$ for which we have to increase $B$ up to $59$ to get $n\leq 5$. The total amount of time for the above computations was roughly $56$ hours.

In Table \ref{table:newform_computations} we give a summary of the data for the spaces of newforms we had to compute together with the amount of time Magma needed to compute the spaces. It is important to note that there are two ways of computing weight $2$ modular forms in Magma, either the classical approach, or using the package of Hilbert newforms viewing classical newforms as Hilbert newforms over $\Q$ \cite{DembeleVoight13}. The package of Hilbert newforms is faster in current implementation of Magma (Magma V2.25-3) and the totally amount of time was roughly $146$ hours with the most expensive case to be the space of level $134400$ and $107$ hours to be computed.
\end{proof}

\begin{table}[]
    \centering
    \begin{tabular}{| c | c | c | c | c |}
        \hline
        Level &  Dimension & \#conjugacy & ($d$, \#newforms of degree $d$) & Time \\
         & & classes & & \\
        \hline
        $2\cdot 5\cdot 7$ & $1$ & $1$ & $(1,1)$ & $\sim 1$ sec\\
        \hline
        $2\cdot 3\cdot 5\cdot 7$ & $5$ & $5$ & $(1,5)$ & $\sim 1$ sec\\
        \hline
        $2\cdot 5^2\cdot 7$ & $10$ & $8$ & $(1,6)$, $(2,4)$ & $\sim 1$ sec\\
        \hline
        $2^3\cdot 3\cdot 5\cdot 7$ & $12$ & $11$ & $(1,10)$, $(2,2)$ & $\sim 3$ sec\\
        \hline
        $2\cdot 3\cdot 5^2\cdot 7$ & $18$ & $18$ & $(1,18)$ & $\sim 4$ sec\\
        \hline
        $2^5\cdot 5\cdot 7$ & $24$ & $20$ & $(1,16)$, $(2,8)$ & $\sim 3$ sec\\
        \hline
        $2^3\cdot 3\cdot 5^2\cdot 7$ & $58$ & $43$ & $(1,32)$, $(2,14)$, $(3,12)$ & $\sim 1$ min\\
        \hline
        $2^5\cdot 5^2\cdot 7$ & $114$ & $52$ & $(1,22)$, $(2,32)$, $(3,12)$, & $\sim 1$ min\\
         & & & $(4,16)$, $(5,20)$, $(6,12)$ & \\
        \hline
        $2^8\cdot 5\cdot 7$ & $192$ & $64$ & $(1,20)$, $(2,24)$, $(3,36)$& $\sim 22$ sec\\
         & & & $(4,16)$, $(6,96)$ & \\
        \hline
        $2^7\cdot 3\cdot 5\cdot 7$ & $192$ & $112$ & $(1,64)$, $(2,56)$, $(3,36)$, & $\sim 2$ min\\
         & & & $(4,16)$, $(5,20)$ & \\
        \hline
        $2^8\cdot 3\cdot 5\cdot 7$ & $384$ & $128$ & $(1,48)$, $(2,32)$, $(3,48)$,& $\sim 30$ min\\
         & & & $(4,112)$, $(6,48)$, $(8,96)$ & \\
        \hline
        $2^8\cdot 5^2\cdot 7$ & $912$ & $196$ & $(1,52)$, $(2,64)$, $(3,36)$, & $\sim 2$ h\\
         & & & $(4,88)$, $(5,40)$, $(6,168)$ & \\
         & & & $(8,96)$, $(9,72)$, $(12,192)$, & \\
         & & & $(16,32)$, $(18,72)$ & \\
        \hline
        $2^7\cdot 3\cdot 5^2\cdot 7$ & $912$ & $356$ & $(1,176)$, $(2,128)$, $(3,36)$ & $\sim 36$ h\\
         & & & $(4,144)$, $(5,140)$, $(6,48)$ & \\
         & & & $(7,168)$, $(9,72)$ & \\
        \hline
        $2^8\cdot 3\cdot 5^2\cdot 7$ & $1824$ & $396$ & $(1,124)$, $(2,120)$, $(3,60),$& $\sim 107$ h\\
         & & & $(4,208)$, $(5,40)$, $(6,240)$, & \\
         & & & $(8,224)$, $(9,72)$, $(10,80)$ & \\
         & & & $(11,88)$, $(12,192)$, $(13,104)$ & \\
         & & & $(16,192)$, $(20,80)$ & \\
        \hline
    \end{tabular}
    \caption{Data for newform computations.}\label{table:newform_computations}
\end{table}

\bibliographystyle{plain}
\bibliography{my_bibliography}

\end{document}